\RequirePackage{fix-cm}
\documentclass[smallextended]{svjour3}       
\smartqed  
\usepackage{graphicx}
\usepackage{cite}
\usepackage{amsmath,amssymb,amsfonts}
\usepackage{algorithmic}
\usepackage{graphicx}
\usepackage{textcomp}
\usepackage{mfl}

\usepackage{ mathrsfs }
\usepackage{setspace}
\usepackage{stmaryrd}
\usepackage{mathabx}

\def\BibTeX{{\rm B\kern-.05em{\sc i\kern-.025em b}\kern-.08em
    T\kern-.1667em\lower.7ex\hbox{E}\kern-.125emX}}

\newtheorem{Def}{\sc Definition}
\newtheorem{Thm}{\sc Theorem}

\newtheorem{Cor}[Thm]{\sc Corollary}

\newtheorem{Pro}[Thm]{\sc Proposition}

\usepackage{color}

\begin{document}

\title{Fra\"iss\'e classes of graded relational structures
}
\subtitle{}


\author{Guillermo Badia        \and
        Carles Noguera 
}


\institute{Guillermo Badia \at
             Department of Knowledge-Based Mathematical Systems \\
        Johannes Kepler University Linz\\
              \email{guillebadia89@gmail.com}            
           \and
Carles Noguera \\
          Institute of Information Theory and Automation \\    
        Czech Academy of Sciences\\
              \email{noguera@utia.cas.cz}  
}

\date{Received: date / Accepted: date}

\maketitle

\begin{abstract}
We study classes of graded structures satisfying the properties of amalgamation, joint embedding and hereditariness. Given appropriate conditions, we can build a graded analogue of the Fra\"iss\'e limit. Some examples such as the class of all finite weighted graphs  or the class of all finite fuzzy orders (evaluated on a particular countable algebra) will be examined.  
\end{abstract}

\section*{Introduction}
In classical model theory any $n$-ary relation $R$ on a universe (or base) $E$ can be seen as a map from $E^n$ into the two-element chain {\bf 2} and a relational structure is simply a collection of relations of any arity with the same universe \cite{Poizat}. 

One can replace the two-element chain by a richer lattice, as done in fuzzy set theory and in many-valued logics, typically based on the interval $[0, 1]$ but also on arbitary lattice-ordered sets \cite{za,g}.  We can call such relations \emph{graded} because they allow elements (or tuples of elements) of the domain satisfy the relation at different levels in a graded scale.
Mathematical fuzzy logic \cite{haj} studies logics of graded relations as particular kinds of many-valued inference systems in several formalisms, including first-order predicate languages. Structures for such first-order graded logics are variations of classical structures in which predicates are interpreted as graded relations. Such structures are relevant for recent computer science developments in which they are called {\em weighted structures} (see e.g.~\cite{HoMoVi:VCSP}).

In the landmark paper \cite{f}, Fra\"iss\'e set out to generalize some properties of some classes of ordered structures. For instance, he observed that the relational structure $\tuple{\mathbb{Q},<}$ is, in a certain sense, the limit of the class of all finite linear orderings and such limit is unique (by virtue of an argument following Cantor's back and forth theorem). Fra\"iss\'e managed to identify the structural properties of the class of finite linear orders that allowed the existence of such a limit by introducing the construction of a structure henceforth known as {\em Fra\"iss\'e limit} \cite{h}  for any class of relational structures satisfying certain suitable properties (a {\em Fra\"iss\'e class}). The core idea here is that in certain circumstances a structure may be built from its finite parts.

The aim of the present paper is to introduce and study Fra\"iss\'e classes and limits for graded structures. In order to put our work in perspective, it is worth mentioning that  the study of the structures of first-order fuzzy logics is based on the corresponding strong completeness theorems \cite{cin,cin1} and has already addressed several crucial topics such as: characterization of completeness properties with respect to models based on particular classes of algebras~\cite{Cintula-EGGMN:DistinguishedSemantics}, models of logics with evaluated syntax~\cite{Novak-Perfilieva-Mockor:2000,MuNo}, study of mappings and diagrams~\cite{de3}, ultraproduct constructions~\cite{Dellunde:RevisitingUltraproducts}, characterization of elementary equivalence in terms of elementary mappings~\cite{de}, L\"owenheim--Skolem theorems~\cite{de2}, and back-and-forth systems for elementary equivalence \cite{de}. Continuous model theory~\cite{Chang-Keisler:ContinuousModelTheory,Caicedo:StrengthLukasiewicz} is a very related area of research that focuses on models over algebras with continuous operations (in this context, Fra\"iss\'e's construction has been studied in~\cite{Ya}).

The paper is organized as follows. In \S \ref{1}, we look in detail at graded structures and related notions such as that of substructures and embeddings. In \S \ref{2}, we provide some examples of classes of graded structures that will be relevant later and we introduce the properties of amalgamation, hereditariness and joint embedding which will be needed for the main theorem of the paper. In  \S \ref{3}, we establish the main of result of the paper: a graded version of the theorem by Fra\"iss\'e showing the construction of the limit structure of suitable classes of structures, besides mentioning a couple of other related facts. In  \S \ref{4}, we prove that some of the classes of structures introduced earlier are in fact suitable for applications of Fra\"iss\'e's theorem; in particular, we present a graded analogue of the random graph which is the Fra\"iss\'e limit of weighted finite graphs evaluated on a fixed countable algebra. Finally, in \S \ref{5} we end with some concluding remarks.

\section{Graded structures}\label{1}
In this section, we begin by introducing in detail the object of our study: graded structures, and several necessary related notions for the development of the paper. For further information the reader may consult the handbook series~\cite{Cintula-FHN:HBMFL}.

We choose, as the underlying algebraic setting, the class of residuated uninorms~\cite{Metcalfe-Montagna:SubstructuralFuzzy}. Most of the well-studied particular systems of fuzzy logic that can be found in the literature can be given a semantics on these algebras (cf.~\cite{Cintula-Noguera-Smith:GradedLORI}).

The algebraic semantics of such logics is based on {\em \UL-algebras}, that is, algebraic structures in the language $\mathcal{L} = \{\wedge,\vee,\conj,\to,\0,\1,\bot,\top\}$ of the form
$\alg{A}=\tuple{A,\wedge^\alg{A},\vee^\alg{A},\conj^\alg{A},\to^\alg{A},\0^\alg{A},\1^\alg{A},\bot^\alg{A},\top^\alg{A}}$ such that
\begin{flushleft}
\begin{itemize}
\item $\tuple{A,\wedge^\alg{A},\vee^\alg{A},\bot^\alg{A},\top^\alg{A}}$ is a bounded lattice,
\item $\tuple{A,\conj^\alg{A},\1^\alg{A}}$ is a commutative monoid,
\item for each $a,b,c\in A$, we have:
\begin{align*}
& a\conj^\alg{A} b\leq c  \quad \mathrm{iff }  \quad  b\leq a \to^\alg{A} c, & \mathrm{(res)}\\
&((a \to^\alg{A} b)\land\1^\alg{A}) \lor^\alg{A} ((b \to^\alg{A} a)\land^\alg{A}\1^\alg{A})  = \1^\alg{A} & \mathrm{(lin)}
\end{align*}
\end{itemize}
\end{flushleft}


$\alg{A}$ is called a {\em \UL-chain} if its underlying lattice is linearly ordered. {\em Standard} \UL-chains are those define over the real unit interval $[0,1]$ with its usual order; in that case the operation $\conj^\alg{A}$ is a residuated uninorm, that is, a left-continuous binary associative commutative monotonic operation with a neutral element $\1^\alg{A}$ (which need not coincide with the value $1$).

Let $\fm$ denote the set of propositional formulas written in the language of \UL-algebras with a denumerable set of variables and let $\FM$ be the absolutely free algebra defined on such set. Given a \UL-algebra $\alg{A}$, we say that an {\em $\alg{A}$-evaluation} is a homomorphism from $\FM$ to $\alg{A}$.  The logic of all \UL-algebras is defined by establishing, for each $\Gamma \cup \{\f\} \subseteq \fm$, $\Gamma \models \f$ if and only if, for each \UL-algebra $\alg{A}$ and each $\alg{A}$-evaluation $e$, we have $e(\f) \geq \1^\alg{A}$, whenever $e(\p) \geq \1^\alg{A}$ for each $\p \in \Gamma$. The logic \UL\ is, hence, defined as preservation of truth over all \UL-algebras, where the notion of truth is given by the set of designated elements, or {\em filter}, $\mathcal{F}^{\alg{A}} = \{a \in A \mid a \geq \1^\alg{A}\}$. The standard completeness theorem of \UL\ proves that the logic is also complete with respect to its intended semantics: the class of \UL-chains defined over $[0,1]$ by residuated uninorms (the standard \UL-chains); this justifies the name of \UL\ (uninorm logic).

Most well-known propositional fuzzy logics can be obtained by extending \UL\  with additional axioms and rules (in a possibly expanded language). Important examples are G\"odel--Dummett logic \G\ and \L ukasiewicz logic $\mathrmL$. 

A {\em predicate language} $\pl$ is a triple $\tuple{\mathbf{P,F,ar}}$, where $\mathbf{P}$ is a non-empty set of predicate symbols, $\mathbf{F}$ is a set of functional symbols, and $\mathbf{ar}$ is a function assigning to each symbol a natural number called the \emph{arity} of the symbol. Let us further fix a denumerable set $V$ whose elements are called {\em object variables}. The sets of {\em $\pl$-terms}, {\em atomic $\pl$-formulas}, and {\em $\tuple{\lang{L},\pl}$-formulas} are defined as usual with the propositional connectives being those of \UL. A {\em \pl-structure} $\model{M}$ is a pair $\tuple{\alg{A},\struct{M}}$ where $\alg{A}$ is a \UL-chain and $\struct{M} = \tuple{M,\Tuple{P_\struct{M}}_{P\in\mathbf{P}},\Tuple{F_\struct{M}}_{F\in\mathbf{F}}}$, where
$M$ is a  set; $P_\struct{M}$ is a function $M^n\to A$, for each $n$-ary predicate symbol $P\in\mathbf{P}$; and $F_\struct{M}$ is a function $M^n\to M$ for each
$n$-ary functional symbol $F\in\mathbf{F}$. An {\em $\model{M}$-evaluation} of the object variables is a mapping $\eval{v} \colon V \to M$; by $\eval{v}[x{\to}a]$ we denote the $\model{M}$-evaluation where $\eval{v}[x{\to}a](x) = a$ and $\eval{v}[x{\to}a](y) = \eval{v}(y)$ for each object variable $y\neq x$. We define the {\em values} of the terms and the {\em truth values} of the formulas as (where for ${\circ}$ stands for any $n$-ary connective in $\lang{L}$):
\begin{center}
\begin{tabular}{rcll}
$\semvalue{x}^\model{M}_\eval{v}$ & $=$ & $\eval{v}(x) ,$
\\[0.5ex]
$\semvalue{F(\vectn{t})}^\model{M}_\eval{v}$ & $=$ & $F_\struct{M}(\semvalue{t_1}^\model{M}_\eval{v}\!,\,\dots
,\,\semvalue{t_n}^\model{M}_\eval{v}),$
\\[0.5ex]
$\semvalue{P(\vectn{t})}^\model{M}_\eval{v}$ & $=$ &
$P_\struct{M}(\semvalue{t_1}^\model{M}_\eval{v}\!,\,\dots ,\,\semvalue{t_n}^\model{M}_\eval{v}),$ &
\\[0.5ex]
$\semvalue{{\circ}(\vectn{\f})}^\model{M}_\eval{v}$ & $=$ & ${\circ}^\alg{A}(\semvalue{\f_1}^\model{M}_\eval{v}\!,\,\dots ,\,\semvalue{\f_n}^\model{M}_\eval{v}),$ &
\\[0.5ex]
$ \semvalue{\All{x}\f}^\model{M}_\eval{v}$ & $=$ & $\inf_{\leq_\alg{A}}
\{\semvalue{\f}^{\model{M}}_{\eval{v}[x{\to}m]} \mid m\in M\},$
\\[0.5ex]
$ \semvalue{\Exi{x}\f}^\model{M}_\eval{v}$ & $=$ & $\sup_{\leq_\alg{A}}
\{\semvalue{\f}^{\model{M}}_{\eval{v}[x{\to}m]} \mid m\in M\}.$
\end{tabular}
\end{center}

If the infimum or supremum does not exist, the corresponding value is
undefined. We say that $\model{M}$ is a {\em safe} if $\semvalue{\f}^{\model{M}}_{\eval{v}}$ is defined for each\/ $\pl$-formula $\f$ and each\/ $\model{M}$-evaluation $\eval{v}$. 

A \emph{model} in this setting would have to refer to a safe structure (in the sense that it would be a model of the given predicate graded logic under consideration), but since the construction of the Fra\"iss\'e limit that we will provide below does not necessarily have to preserve safety, our main concern here are simply structures. However, if the algebra of the model is finite, our construction will certainly preserve safety and we can talk about models again.

An important caveat is necessary at this point:  from now on, we will have a fixed \UL-chain $\alg{A}$ for the purposes of the construction in this paper. Therefore all the graded structures in this paper are assumed to be valued on $\alg{A}$.

The main reason for the above restriction is that allowing the algebra to vary arbitrarily gives us a main theorem that simply follows as a corollary of a many-sorted  Fra\"iss\'e theorem. This is because then we can read graded structures as two-sorted classical structures (see~\cite{Cintula-EGGMN:DistinguishedSemantics,de2}) and frame everything in terms of the standard Fra\"iss\'e results. Such a result would not have the same interest from the point of view of graded model theory, where one usually wants to fix a particular intended algebra of truth-values. 

Let us now recall the notion of  substructure (see e.g.~\cite{de2}).

\begin{Def}
$\tuple{\alg{A},\struct{M}}$ is a \emph{substructure} of $\tuple{\alg{A},\struct{N}}$ if the following conditions are satisfied:
\begin{enumerate}
\item $M\subseteq N$.
\item For each $n$-ary functional symbol $F\in \mathbf{F}$, and elements $\vectn{d}\in M$, $$F_{\struct{M}}(\vectn{d})=F_{\struct{N}}(\vectn{d}).$$

\item For every quantifier-free formula $\f(\vectn{x})$, and $\vectn{d} \in M$, $$\semvalue{\f(\vectn{d})}_\struct{M}^\alg{A}= \semvalue{\f(d_1,\ldots,d_n)}_\struct{N}^\alg{A}.$$
\end{enumerate}

$\tuple{\alg{A},\struct{M}}$ is \emph{generated} if $M$ has a set of generators in the obvious sense. Moreover, if such set of generators is finite, $\tuple{\alg{A},\struct{M}}$ is said to be \emph{finitely generated}.
\end{Def}

A structure $\tuple{\alg{A}, \struct{M}}$ will be said to be \emph{countable} if $M$ is countable.

\begin{Def} A pair of maps  $\tuple{g, f}$ is an \emph{embedding} from $\tuple{\alg{A},\struct{M}}$ into $\tuple{\alg{A},\struct{N}}$ if the following conditions are satisfied:
\begin{enumerate}
\item $f\colon M \longrightarrow N$ is one-to-one.
\item For each $n$-ary functional symbol $F\in \mathbf{F}$, and elements $\vectn{d}\in M,$ $$f(F_{\struct{M}}(\vectn{d}))=F_{\struct{N}}(f(d_1),\ldots,f(d_n)).$$

\item For every quantifier-free formula $\f(\vectn{x})$, and $\vectn{d} \in M$, $$\semvalue{\f(\vectn{d})}_\struct{M}^\alg{A}= \semvalue{\f(f(d_1),\ldots,f(d_n))}_\struct{N}^\alg{A}.$$

\item g is the identity map on $\alg{A}$.
\end{enumerate}
\end{Def}

In what follows, by an \emph{isomorphism} we will simply mean an onto embedding between two structures. 

A sequence $\{\tuple{\alg{ A}, {\bf M}_i} \mid i < \gamma\}$ of structures is a called a \emph{chain} when for all $i<j<\gamma$ we have that $\tuple{\alg{ A} , {\bf M}_i}$ is a substructure of  $\tuple{\alg{ A}, {\bf M}_j}$. The \emph{union} of the chain $\{\tuple{\alg{ A}, {\bf M}_i} \mid i < \gamma\}$ is the structure $\tuple{\alg{ A}, \bigcup_{i < \gamma} {\bf M}_i}$ where {\bf M} is defined by taking as its domain $\bigcup _{i<\gamma}{ M}_i$, interpreting the constants and functionals of the language as they were interpreted in each ${\bf M}_i$ and similarly with the relational symbols of the language. Observe as well that ${\bf M}$ is well-defined given that $\{\tuple{\alg{ A}, {\bf M}_i} \mid i < \gamma\}$ is a chain.

\section{Some classes of graded structures}\label{2}
Let us introduce some useful examples of graded structures. By cardinality reasons, it is necessary to fix a countable UL-chain $\alg{A}$ for the structures; otherwise, with an uncountable algebra, the number of possible finite structures even on a finite language grows beyond countable.

Let $ \mathscr{K}_0$ be the class of all finite $\alg{A}$-structures $\tuple{\alg{A}, \struct{M}}$ where there is only one binary relation $<$ and for $a, b, c \in M$ (``pre-orders''):

\begin{itemize}
\item[(0.1)] $\semvalue{a < a}_{\struct{M}}^{\alg{A}}  \geq \1^\alg{A}$

\item[(0.2)] $\semvalue{(a<b \wedge b<c) \rightarrow a<c }_{\struct{M}}^{\alg{A}}  \geq \1^\alg{A}$
\end{itemize}

If we let \alg{A} be finite, we could also describe this class as the collection of all finite $\alg{A}$-structures $\tuple{\alg{A}, \struct{M}}$ where
\begin{itemize}
\item[(0.1)$^{\prime}$] $\semvalue{\All{x} (x < x)}_{\struct{M}}^{\alg{A}}  \geq \1^\alg{A}$

\item[(0.2)$^{\prime}$] $\semvalue{\All{x}\All{y}\All{z} ((x<y \wedge y<z) \rightarrow x<z)}_{\struct{M}}^{\alg{A}} \geq \1^\alg{A}$
\end{itemize}

Let $ \mathscr{K}_1$ define the class of all finite $\alg{A}$-structures $\tuple{\alg{A}, \struct{M}}$ where there is only one binary relation $<$ and for $a, b, c \in M$ (``\alg{A}-weighted graphs''):

\begin{itemize}
\item[(1.1)] $\semvalue{a < a}_{\struct{M}}^{\alg{A}} < \1^\alg{A}$

\item[(1.2)] $\semvalue{a<b \rightarrow b<a }_{\struct{M}}^{\alg{A}} \geq \1^\alg{A}$
\end{itemize}

Now, if we let $\alg{A}$ be finite and we expand the language $\pl$ with a collection of constants for each element of $\alg{A}$ to a language  $\pl^\alg{A}$, we can define this class  as the collection of all finite $\alg{A}$-structures $\tuple{\alg{A}, \struct{M}}$ where

\begin{itemize}
\item[(1.1)$^{\prime}$] $ \semvalue{ \All{x} (x < x)  \rightarrow d_a}_{\struct{M}}^{\alg{A}} \geq \1^\alg{A}  $, where $d_a$ is the immediate predecessor of $\1^\alg{A}$ in the linear order of \alg{A}.

\item[(1.2)$^{\prime}$] $ \semvalue{\All{x} \All{y}(x<y \rightarrow y<x) }_{\struct{M}}^{\alg{A}}   \geq \1^\alg{A}$
\end{itemize}

Let $ \mathscr{K}_2$ be the class of all finite $\alg{A}$-structures $\tuple{\alg{A}, \struct{M}}$ where there is only one binary relation $<$ and for $a, b, c \in M$ (``total-orders''):

\begin{itemize}
\item[(2.1)] $\semvalue{a < a}_{\struct{M}}^{\alg{A}}  \geq \1^\alg{A} $

\item[(2.2)] $\semvalue{(a<b \wedge b<c) \rightarrow a<c}_{\struct{M}}^{\alg{A}}  \geq \1^\alg{A}$

\item[(2.2)] $\semvalue{ a<b \vee b<a}_{\struct{M}}^{\alg{A}}  \geq \1^\alg{A}$
\end{itemize}

If we let \alg{A} be finite, we could also describe this class as the collection of all finite $\alg{A}$-structures $\tuple{\alg{A},\struct{M}}$ where
\begin{itemize}
\item[(2.1)$^{\prime}$] $\semvalue{\All{x}(x < x)}_{\struct{M}}^{\alg{A}}  \geq \1^\alg{A} $

\item[(2.2)$^{\prime}$] $\semvalue{\All{x} \All{y} \All{z} ((x<y \wedge y<z) \rightarrow x<z)}_{\struct{M}}^{\alg{A}}  \geq \1^\alg{A}$

\item[(2.2)$^{\prime}$] $\semvalue{ \All{x} \All{y} (x<y \vee y<x)}_{\struct{M}}^{\alg{A}}  \geq \1^\alg{A}$
\end{itemize}

Let $\mathscr{K}_3$ be the class of all finite $\alg{A}$-structures $\tuple{\alg{A}, \struct{M}}$ where there is only one binary relation $<$ and for $a, b, c \in M$:

\begin{itemize}
\item[(3.1)] $\semvalue{a < a}_{\struct{M}}^{\alg{A}}  \geq \1^\alg{A}  $

\item[(3.2)] $\semvalue{a < b }_{\struct{M}}^{\alg{A}}  \geq \1^\alg{A} $ and $\semvalue{b < c }_{\struct{M}}^{\alg{A}}  \geq \1^\alg{A}  $ only if  $\semvalue{a < c }_{\struct{M}}^{\alg{A}}  \geq \1^\alg{A}   $

\item[(3.3)]   $\semvalue{ a < b }_{\struct{M}}^{\alg{A}}  \geq \1^\alg{A} $ and $|| b < a ||_{\struct{M}}^{\alg{A}} \geq \1^\alg{A} $ only if $a = b$
\end{itemize}

Next we need to introduce three properties of classes of structures that will play a fundamental role in the main theorem of the paper.

\begin{Def}
A class $\mathscr{K}$ of  relational $\alg{A}$-structures is said to have the \emph{hereditary property (HP)} if  $\mathscr{K}$ is closed under taking substructures. 
\end{Def}

\begin{Def}
A class $\mathscr{K}$ of structures is said to have the \emph{joint embedding property (JEP)} if any two elements of $\mathscr{K}$ have a common extension in $\mathscr{K}$.
\end{Def}

\begin{Def}
 We say that  $\mathscr{K}$ has the { \em amalgamation property (AP)} if given a v-formation $\tuple{\alg{A},\struct{M}_0} \subseteq \tuple{\alg{A},\struct{M}_1}$, $\tuple{\alg{A},\struct{M}_0} \subseteq \tuple{\alg{A}, \struct{M}_2}$ of structures in $\mathscr{K}$, there is $\tuple{\alg{A}, \struct{M}_3} \in \mathscr{K}$ and embeddings $\tuple{Id, f_1}\colon \tuple{\alg{A},\struct{M}_1} \longrightarrow \tuple{\alg{A}, \struct{M}_3}$ and $\tuple{Id,f_2}\colon \tuple{\alg{A},\struct{M}_2} \longrightarrow \tuple{\alg{A},\struct{M}_3}$ which coincide on their images for the elements of $\tuple{\alg{A}, \struct{M}_0}$.
\end{Def}

For any class of graded structures $\mathscr{K}$, by $\mathscr{K}^{\cong}$ we will mean the class of isomorphism types of $\mathscr{K}$, that is, for every element $\tuple{\alg{A},\struct{M}}$ of $\mathscr{K}$, $\mathscr{K}^{\cong}$ will contain exactly one structure isomorphic to $\tuple{\alg{A},\struct{M}}$. 

\section{Fra\"iss\'e's theorem for classes of graded structures}\label{3}
In this section we are ready to establish the main result of the paper regarding the construction of a structure from its finitely generated parts. We begin with an auxiliary definition.

\begin{Def}
The \emph{age} of a structure $\tuple{\alg{A}, \struct{M}}$, in symbols $\mbox{Age}(\alg{A}, \struct{M})$, is the collection of all finitely generated substructures of  $\tuple{\alg{A}, \struct{M}}$ and their isomorphic copies.
\end{Def}

In practice, when speaking about the age of $\tuple{\alg{A}, \struct{M}}$ we simply mean the collection of isomorphism types of its age (otherwise, the class grows too big and unmanageable).

\begin{Thm}\label{thm1} Let $\mathscr{K}$ be a countable class of finitely generated $\alg{A}$-structures for the same language $\pl$.  Then, $ \mathscr{K} =  \mbox{Age}(\alg{A}, \struct{N})$ for some $\tuple{\alg{A}, \struct{N}}$ iff $\mathscr{K}$ satisfies HP and JEP.  Furthermore, if $ \mathscr{K}$ is a class of arbitrary cardinality which  satisfies HP, JEP and is closed under unions of chains, then $\mathscr{K} =  \mbox{Age}(\alg{A}, \struct{N})$ for some structure $\tuple{\alg{A}, \struct{N}}$.\end{Thm}

\begin{proof} First, suppose that $ \mathscr{K} =  \mbox{Age}(\alg{A}, \struct{N})$ for some $\tuple{\alg{A}, \struct{N}}$. Then, $ \mathscr{K}$ must satisfy HP because any finitely generated substructure of a structure in the age of $\tuple{\alg{A}, \struct{N}}$ must remain in the age. Now,  given $\tuple{\alg{A}, \struct{M}_1}, \
\tuple{\alg{A}, \struct{M}_2} \in \mbox{Age}(\alg{A}, \struct{N})$, we may consider the structure generated by the finite union of the generators of $\tuple{\alg{A}, \struct{M}_1}$ and $\tuple{\alg{A}, \struct{M}_2}$, which is obviously also in $\mbox{Age}(\alg{A}, \struct{N})$. Hence, $\mbox{Age}(\alg{A}, \struct{N})$ has JEP.

Conversely, assume that $\mathscr{K}$ satisfies HP and JEP. Since $\mathscr{K}$ is countable we can take an enumeration of its members: $\tuple{\alg{A}_0, \struct{M}_0}, \tuple{\alg{A}_1, \struct{M}_1}, \dots$. We define inductively a chain of elements of $ \mathscr{K}$, $\tuple{\alg{A}, \struct{N}_0}, \tuple{\alg{A}, \struct{N}_1}, \dots$ as follows: 
$\tuple{\alg{A}, \struct{N}_0} = \tuple{\alg{A}, \struct{M}_0}$, and given $\tuple{\alg{A}, \struct{N}_i} $ we let $\tuple{\alg{A}, \struct{N}_{i+1}} $ be obtained by JEP with $\tuple{\alg{A}, \struct{M}_{i+1}} $. Finally, we take the union of the chain we have constructed and check that its age is exactly $\mathscr{K}$; indeed, every member of $\mathscr{K}$ is certainly in the age of this union and, conversely, if some structure is in the age, it must be in $\mathscr{K}$ by construction and thanks to HP. 

Now, if $\mathscr{K}$ is an arbitrary class satisfying HP and JEP, and closed under unions of chains, we can repeat the construction in the previous paragraph and deal with the case of limit ordinals by taking unions.\qed
\end{proof}

 \begin{Cor} Let  $\mathscr{K}$ be a class of structures of same language $\pl$. Then $\mathscr{K} = \{ \tuple{\alg{A}, \struct{M}} \mid \ \mbox{Age}(\alg{A}, \struct{M}) \subseteq  \mbox{Age}(\alg{A}, \struct{N})\}$ for some $\tuple{\alg{A}, \struct{N}}$ if $\mathscr{K}$ satisfies HP and JEP and is closed under unions of chains. \end{Cor}

\begin{proof}
Suppose  that $\mathscr{K}$ satisfies HP and JEP and, moreover, it is closed under unions of chains. The collection of finitely generated substructures of structures from $\mathscr{K}$ (denoted as  $\mathscr{K} ^{\prime}$)  satisfies HP rather trivially since $\mathscr{K} $ itself satisfies it. On the other hand,  $\mathscr{K} ^{\prime}$ also satisfies JEP, because given $\tuple{\alg{A}, \struct{M}_1}, \tuple{\alg{A}, \struct{M}_2} \in \mathscr{K}^{\prime}$, using JEP for $\mathscr{K} $, we obtain $\tuple{\alg{A}, \struct{M}_3} \in \mathscr{K} $ containing $\tuple{\alg{A}, \struct{M}_1}, \tuple{\alg{A}, \struct{M}_2}$ as finitely generated substructures, and then we can generate a substructure of $\tuple{\alg{A}, \struct{M}_3}$ in $\mathscr{K} ^{\prime}$ from the finite union of the generators of $\tuple{\alg{A}, \struct{M}_1}$ and $\tuple{\alg{A}, \struct{M}_2}$. Then, $\mathscr{K} ^{\prime} = \mbox{Age}(\alg{A}, \struct{N})$. Now, if $\tuple{\alg{A}, \struct{M}} \in \mathscr{K} $ and $\tuple{\alg{A}, \struct{M}^{\prime}} \subseteq \tuple{\alg{A}, \struct{M}}$ is finitely generated, then $\tuple{\alg{A}, \struct{M}^{\prime}} \in \mathscr{K} ^{\prime} = \mbox{Age}(\alg{A}, \struct{N})$.

 On the other hand, assume that $\tuple{\alg{A}, \struct{M}^{\prime}} \in  \mbox{Age}(\alg{A}, \struct{N}) = \mathscr{K} ^{\prime}$ for all finitely generated substructures of $\tuple{\alg{A}, \struct{M}}$. We need to show that, in fact, $\tuple{\alg{A}, \struct{M}} \in  \mathscr{K} $. This can be accomplished by  induction on the cardinality of the generators of the base of the structure $\tuple{\alg{A}, \struct{M}} $. 

First suppose that the set of generators of $M$ is finite. Then, trivially,  $\tuple{\alg{A}, \struct{M}}\in  \mbox{Age}(\alg{A}, \struct{N}) = \mathscr{K} ^{\prime} \subseteq \mathscr{K} $. Next, let the set of generators of $M$ be infinite, of cardinality $\lambda$ (and by inductive hypothesis assume that we have the result holding when the set of generators of $M$ has cardinality $< \lambda$).  We may suppose that $\tuple{\alg{A}, \struct{M}^{\prime}} \in  \mbox{Age}(\alg{A}, \struct{N}) = \mathscr{K} ^{\prime}$ for all finitely generated substructures of $\tuple{\alg{A}, \struct{M}}$ only if $\tuple{\alg{A}, \struct{M}} \in   \mathscr{K} $, whenever $M$ is generated by $< \lambda$ elements. We can, however, write $\tuple{\alg{A}, \struct{M}}$ as the union of a chain of substructures $\tuple{\alg{A}, \struct{M}^{\prime}}$ such that the set of generators of $M^{\prime}$ has cardinality $< \lambda$.\qed
\end{proof}

We can apply the characterization to the first of our examples: the class of finite pre-orders.

\begin{Pro}  $ \mathscr{K}_0^{\cong}$ is the age of some structure.\end{Pro}

\begin{proof}
 $ \mathscr{K}_0^{\cong}$ is countable and it  has HP because the property of being a preorder will be preserved under taking substructures. On the other hand, it is not difficult to show that it has JEP. Take  $\tuple{\alg{A}, \struct{M}_1}, \tuple{\alg{A}, \struct{M}_2} \in \mathscr{K}_0^{\cong}$ (and we may suppose that their bases $M_1$ and $M_2$ are disjoint). Build the $\alg{A}$-structure $\tuple{\alg{A}, \struct{M}_3}$ where the base of $\struct{M}_3$ is the union $M_1 \cup M_2$ and the relations take values according to $\tuple{\alg{A}, \struct{M}_1}$ and $ \tuple{\alg{A}, \struct{M}_2}$ except when one of the arguments comes from $M_1$ and the other from $M_2$, in which case we evaluate the relation for that pair as $\0^\alg{A}$. Now, $\1^\alg{A}$ is a lower bound of $\{ \semvalue{a < a }_{\struct{M}_3}^{\alg{A}} \mid a \in M_3\}$ in $\alg{A}$ because $\1^\alg{A}$ is a lower bound  of $\{ \semvalue{ a < a }_{\struct{M}_i}^{\alg{A}} \mid a \in M_i\}$ in $\alg{A}$ ($i = 1, 2$). Similarly, $ \1^\alg{A}$ is a lower bound of 
 $$
\{ \semvalue{( a<b \wedge b<c ) \rightarrow a<c }_{\struct{M}_3}^{\alg{A}} \mid a, b, c \in M_3\}
$$
 in $\alg{A}$. Hence, by Theorem~\ref{thm1}, this class is the age of some structure.\qed
\end{proof}

The property described in the following definition is sometimes also called \emph{ultrahomogeneity} \cite{h}, but we can use the original name since there will be no other notion of homogeneity in this paper.
\begin{Def}
A relational structure is called \emph{homogeneous} if every isomorphism between two finitely generated substructures extends to an automorphism of the structure.
\end{Def}

The structure constructed in the next theorem will be called the \emph{Fra\"iss\'e limit} of the class $\mathscr{K}$ in question.  By a \emph{Fra\"iss\'e class} we mean any class of graded structures satisfying the properties described in the theorem. The improvement over Theorem~\ref{thm1} is that this time we will make the class in question the age of a unique structure which will be homogeneous. This is accomplished by demanding that the class satisfies AP.

\begin{Thm}\label{thm2}\emph{(Fra\"iss\'e's theorem)} Let $\mathscr{K}$ be a countable  set of finitely generated structures  of the same language $\pl$. If $\mathscr{K}$ has HP, JEP, and AP, then there is a unique countable homogeneous structure $\tuple{\alg{A}, \struct{M}}$ such that $\mathscr{K} = \mbox{Age}(\alg{A}, \struct{M})$ (up to isomorphism).  Moreover, if a structure is a homogeneous, then its age has AP. \end{Thm}

\begin{proof} Let us construct inductively a chain $\tuple{\alg{A}, \struct{M}_i}_{i < \omega}$ of elements of $\mathscr{K}$  such that if $\tuple{\alg{A}, \struct{N}}, \tuple{\alg{A}, \struct{N}^{\prime}}\in \mathscr{K}$, $\tuple{\alg{A}, \struct{N}} \subseteq \tuple{\alg{A}, \struct{N}^{\prime}}$ and there is an embedding $$\tuple{Id, f} \colon \tuple{\alg{A}, \struct{N}} \longrightarrow \tuple{\alg{A}, \struct{M}_i}$$ for some $i$, then there exists an embedding $$\tuple{Id, f^{\prime}} \colon \tuple{\alg{A}, \struct{N}^{\prime}} \longrightarrow \tuple{\alg{A}, \struct{M}_j}$$ for some $j > i$ extending $\tuple{Id, f}$. Let $\tuple{\alg{A}, \struct{M}_0} \in \mathscr{K}$ be arbitrary. Now, given $\tuple{\alg{A}, \struct{M}_i}$, we can list as $$\langle \tuple{Id, f_{ij}}, \tuple{\alg{A}, \struct{N}_{ij}}, \tuple{\alg{A}, \struct{N}^{\prime}_{ij}} \rangle_{j < \omega}$$  all the triples such that $\tuple{\alg{A}, \struct{N}_{ij}} \subseteq \tuple{\alg{A}, \struct{N}^{\prime}_{ij}}$ and $$\tuple{Id, f_{ij}} \colon \tuple{\alg{A}, \struct{N}_{ij}} \longrightarrow \tuple{\alg{A}, \struct{M}_i}.$$ Now we may build an auxiliary chain $\tuple{\alg{A}, \struct{M}_{ij}}_{j < \omega}$ inductively as follows. First, put $\tuple{\alg{A}, \struct{M}_{i0}} = \tuple{\alg{A}, \struct{M}_{i}} $. Next, having defined $\tuple{\alg{A}, \struct{M}_{ij}}$, we obtain $\tuple{\alg{A}, \struct{M}_{ij+1}}$ by amalgamation with  $\tuple{\alg{A}, \struct{N}_{ij}}$ and $ \tuple{\alg{A}, \struct{N}^{\prime}_{ij}} $ (this can be done since we can take isomorphic copies of structures which will turn embeddings into substructure relations  in order to apply AP). Let $\tuple{\alg{A}, \struct{M}_{i+1}}$ be the union of $\tuple{\alg{A}, \struct{M}_{ij}}_{j < \omega}$. 

Now consider the union of $\tuple{\alg{A}, \struct{M}_i}_{i < \omega}$. The age of this structure is certainly included in $\mathscr{K}$, because any finitely generated substructure would have to be a finitely generated substructure of some member of the chain, which is in $\mathscr{K}$, and by HP we have what we desire. On the other hand, if $\tuple{\alg{A}, \struct{N}}$ is in $\mathscr{K}$, using the JEP with $\tuple{\alg{A}, \struct{M}_0}$, we can produce $\tuple{\alg{A}, \struct{N}^{\prime}} \supseteq \tuple{\alg{A}, \struct{M}_0}, \tuple{\alg{A}, \struct{N}}$. Now considering the identity embedding from $\tuple{\alg{A}, \struct{M}_0}$ into itself, we can see, by the property of the union of $\tuple{\alg{A}, \struct{M}_i}_{i < \omega}$ that we ensured by construction, that for some $i > 0$, there exists an embedding $$\tuple{Id, f} \colon\tuple{\alg{A}, \struct{N}^{\prime}} \longrightarrow \tuple{\alg{A}, \struct{M}_i},$$ so indeed, $\tuple{\alg{A}, \struct{N}}$ is in the age of the union of $\tuple{\alg{A}, \struct{M}_i}_{i < \omega}$.

Finally, the union of $\tuple{\alg{A}, \bigcup_{i < \omega}\struct{M}_i}$ must also be homogeneous. We prove this next.

Let $\tuple{\alg{A}, \struct{N}} \subseteq \tuple{\alg{A}, \bigcup_{i < \omega}\struct{M}_i}$ be finitely generated. Then  $\tuple{\alg{A}, \struct{N}} \subseteq \tuple{\alg{A}, \struct{M}_i}$ for some $i$, so we may find an isomorphism $$\tuple{Id, f} \colon\tuple{\alg{A}, \struct{M}_i} \longrightarrow \tuple{\alg{A}, \struct{M}_j^{\prime}}$$ for a finitely generated $ \tuple{\alg{A}, \struct{M}_j^{\prime}} \subseteq  \tuple{\alg{A}, \struct{M}_j}$ and some $j > i$ extending the identity on   $\tuple{\alg{A}, \struct{N}}$. Enumerate the elements of $\bigcup_{i < \omega}{ M}_i \setminus  M_i$  as $x_1, x_2, \dots$  Enumerate the elements of $\bigcup_{i < \omega}{ M}_i \setminus {M}_j^{\prime}$  as $x_1^{\prime}, x_2^{\prime}, \dots$

Define a chain of isomorphisms between finitely generated substructures of  $\tuple{\alg{A}, \bigcup_{i < \omega}\struct{M}_i}$
$$\tuple{Id, f_0} \colon\tuple{\alg{A}, \struct{M}_i}_0 \longrightarrow \tuple{\alg{A}, \struct{M}_j^{\prime}}_0, \tuple{Id, f_1} \colon\tuple{\alg{A}, \struct{M}_i}_1 \longrightarrow \tuple{\alg{A}, \struct{M}_j^{\prime}}_1, \dots
$$
such that $\tuple{\alg{A}, \struct{M}_i}_k$ contains the substructure of $\tuple{\alg{A}, \bigcup_{i < \omega}\struct{M}_i}$ generated by the finite generators of $ \tuple{\alg{A}, \struct{M}_i}$ together with $\{x_1, \dots, x_k\}$ (and similarly for $\tuple{\alg{A}, \struct{M}_j^{\prime}}_k$ and $\{x_1^{\prime}, \dots, x_k^{\prime}\}$) inductively as follows.

{\sc Stage} 0:  Let $\tuple{Id, f_0} = \tuple{Id,f}$.

 {\sc Stage} $ k+ 1$:  Assume now that we have been given $$\tuple{Id, f_{k}} \colon \tuple{\alg{A}, \struct{M}_i}_k \longrightarrow \tuple{\alg{A}, \struct{M}_j^{\prime}}_k.$$ The structure $ \tuple{\alg{A}, \struct{M}_i}_k $   is in $\mathscr{K}$ and is contained in the structure  $ \tuple{\alg{A}, \struct{M}_i}_k^{\prime} $ generated by the finite generators of $ \tuple{\alg{A}, \struct{M}_i}_k$ plus  $\{x_1, \dots, x_{k+1}\}$, so we may find an isomorphism $\tuple{Id, f_{k}^{\prime}} $ from  $ \tuple{\alg{A}, \struct{M}_i}_k^{\prime} $ to some finitely generated extension $\tuple{\alg{A}, \struct{M}_j^{\prime}}_k^{\prime}$
 of $\tuple{\alg{A}, \struct{M}_j^{\prime}}_k$. Similarly, we may consider $\tuple{Id, f_{k}^{\prime -1}} $ to get an isomorphism $\tuple{Id, f_{k}^{\prime -1 \prime}}$ from  the substructure $\tuple{\alg{A}, \struct{M}_j^{\prime}}_k^{\prime \prime}$ of $\tuple{\alg{A}, \bigcup_{i < \omega}\struct{M}_i}$ generated by the finite generators of $\tuple{\alg{A}, \struct{M}_j^{\prime}}_k^{\prime}$ together with $\{x_1^{\prime}, \dots, x_{k+1}^{\prime}\}$ to some finitely generated extension $ \tuple{\alg{A}, \struct{M}_i}_k^{\prime \prime} $ of $ \tuple{\alg{A}, \struct{M}_i}_k^{\prime} $. Finally we put $\tuple{Id, f_{k}^{\prime -1 \prime -1}} = \tuple{Id, f_{k+1}}$, $ \tuple{\alg{A}, \struct{M}_i}_k^{\prime \prime} =  \tuple{\alg{A}, \struct{M}_i}_{k+1}$ and $\tuple{\alg{A}, \struct{M}_j^{\prime}}_k^{\prime \prime} = \tuple{\alg{A}, \struct{M}_j^{\prime}}_{k+1}$.

The union of this chain of mappings provides the desired automorphism of $\tuple{\alg{A}, \bigcup_{i < \omega}\struct{M}_i}$ extending the inclusion on $\tuple{\alg{A}, \struct{N}}$. 

Now let us establish that $\tuple{\alg{A}, \bigcup_{i < \omega}\struct{M}_i}$ is unique. For suppose that $\tuple{\alg{A}, \struct{N}}$ is another such structure.  We will construct a chain $\tuple{Id, f_{n}}_{n < \omega}$ of isomorphisms between finitely generated substructures of  $\tuple{\alg{A}, \bigcup_{i < \omega}\struct{M}_i}$ and $\tuple{\alg{A}, \struct{N}}$ and we will consider the union of such chain as our isomorphism. To this purpose just proceed as in the proof of homogeneity, this time splitting the successor steps between even and odd stages. At odd stages make sure that the domain of the final map will include the totality of the elements of $\tuple{\alg{A}, \bigcup_{i < \omega}\struct{M}_i}$, whereas at even stages make sure that the range of the map will include the totality of the elements of $\tuple{\alg{A}, \struct{N}}$. The isomorphism at stage 0 this time comes from the fact that by assumption $\tuple{\alg{A}, \struct{N}}$ and $\tuple{\alg{A}, \bigcup_{i < \omega}\struct{M}_i}$ have the same age.

For the second part of the theorem, suppose that $\tuple{\alg{A}, \struct{M}}$ is a homogeneous structure.
Consider  a v-formation $$\tuple{\alg{A}, \struct{M}_0} \subseteq \tuple{\alg{A}, \struct{M}_1} \,\,\,\,\,\, \tuple{\alg{A}, \struct{M}_0} \subseteq \tuple{\alg{A}, \struct{M}_2}$$ in $\mbox{Age}(\alg{A}, \struct{M})$. We need to find $ \tuple{\alg{A}, \struct{M}_3} \in \mbox{Age}(\alg{A}, \struct{M})$ and embeddings $$\tuple{Id, f_1}\colon\tuple{\alg{A}, \struct{M}_1} \longrightarrow \tuple{\alg{A}, \struct{M}_3}$$ and $$\tuple{Id, f_2}\colon\tuple{\alg{A}, \struct{M}_2} \longrightarrow \tuple{\alg{A}, \struct{M}_3}$$ which coincide on their images for the elements of $\tuple{\alg{A}, \struct{M}_0}$. But then we can extend the isomorphisms from $\tuple{\alg{A}, \struct{M}_0} $ into a substructure of $\tuple{\alg{A}, \struct{M}_1}$  and from $\tuple{\alg{A}, \struct{M}_0} $ into a substructure of $\tuple{\alg{A}, \struct{M}_2}$ to automorphisms of $\tuple{\alg{A}, \struct{M}}$, $\tuple{Id, g_1}$ and $ \tuple{Id, g_2}$. The amalgam $\tuple{\alg{A}, \struct{M}_3}$ will come from considering the structure with universe $g_1^{-1}(M_1)\cup g_2^{-1}(M_2)$ and evaluation induced by $\tuple{\alg{A}, \struct{M}_1}$ and $\tuple{\alg{A}, \struct{M}_2}$.\qed
\end{proof}

\section{Examples of Fra\"iss\'e classes}\label{4}

A {\em weighted graph} (or {\em $\alg{A}$-weighted graph} as we have called them before) is like a standard graph except that each edge has an associated value from some algebra $\alg{A}$.  Let us describe their Fra\"iss\'e limit.

\begin{Pro} $ \mathscr{K}_1^{\cong}$  is a Fra\"iss\'e class. \end{Pro}

\begin{proof} All the properties follow easily. In particular, HP follows because the properties defining a weighted graph are preserved under the substructure construction. JEP follows by considering the simple union of two weighted graphs defined in the obvious way. AP follows by the same construction.  \qed \end{proof}

\begin{Thm} Let $\tuple{\alg{A}, \struct{M}}$ be a countable weighted graph (where $\alg{A}$ is also countable).  Then the following are equivalent.
\begin{itemize}
\item[(i)]  $\tuple{\alg{A}, \struct{M}}$ is the Fra\"iss\'e limit of $ \mathscr{K}_1$. 

\item[(ii)]  $\tuple{\alg{A}, \struct{M}}$ is the random $\alg{A}$-weighted graph: whenever we have a map $f\colon X \longrightarrow \alg{A}$, where $X \subseteq M$ and $|X| < \omega$,  we can find a vertex $w$ in $\tuple{\alg{A}, \struct{M}}$ such that for each $a \in A$, there are edges connecting the elements of the fiber of $f$ over $a$ to $w$ with weight $a$. 

\end{itemize}
\end{Thm}

\begin{proof} $(i) \implies (ii)$: Consider a map $f \colon X \longrightarrow \alg{A}$, where $X \subseteq M$ and $|X| < \omega$. We can build a finite $\alg{A}$-weighted  $\tuple{\alg{A}, \struct{N}}$  as follows: take a new vertex $v$ and let $M = X \cup \{v\}$. The graded relation $R$ of this weighted graph will be such that $\semvalue{Rab}^\alg{A}_\struct{N} = \semvalue{Rab}^\alg{A}_\struct{M}  $ if $a, b \in X$ and $\semvalue{Rav}^\alg{A}_\struct{N} = \semvalue{Rva}^\alg{A}_\struct{N} = f(a) $ for $a \in X$. But then since  $\tuple{\alg{A},\struct{M}}$ is a Fra\"iss\'e limit, we have the existence of an embedding $$\tuple{Id,g} \colon  \tuple{\alg{A}, \struct{N}}  \longrightarrow \tuple{\alg{A},\struct{M}},$$ and hence the restriction of $\tuple{Id, g}$ to  $ \tuple{\alg{A},{\bf X}}$ is an isomorphism between finitely generated substructures of  $\tuple{\alg{A},\struct{M}}$ so it extends to an automorphism $\tuple{Id, g^{\prime}}$ of $\tuple{\alg{A},\struct{M}}$. Then $g^{\prime -1}g(v)$ will be the desired element of $\tuple{\alg{A}, \struct{M}}$. 

$(ii) \implies (i)$: We want to establish the following:

\begin{itemize}
\item[(a)] Take finite  $\alg{A}$-weighted graphs $\tuple{\alg{A}, {\bf G}} \subseteq \tuple{\alg{A}, {\bf H}} $ and consider an embedding $\tuple{Id, g} \colon  \tuple{\alg{A}, {\bf G}}  \longrightarrow \tuple{\alg{A}, \struct{M}} $. Then, we can extend $\tuple{Id,g}$ to an  embedding $\tuple{Id, g^{\prime}} \colon  \tuple{\alg{A}, {\bf H}}  \longrightarrow \tuple{\alg{A}, \struct{M}} $.
\end{itemize}

To prove (a) we proceed by induction on the number $n$ of elements in $H$ not in $G$. Clearly we only need to concern ourselves with the case $n = 1$. Consider the vertex $v$ which is in $H$ but not in $G$, then let $X$ be the collection of vertices $a$ from $G$ such that there is an edge with some weigth assigned between $v$ and $a$ in $\tuple{\alg{A}, {\bf H}}$. Now take $h\colon g[X] \longrightarrow \alg{A}$ to be such that $h(g(x))$  is simply the weight in  $\tuple{\alg{A}, {\bf H}}$ of the edge $\{x, v\}$. We build $\tuple{Id, g^{\prime}}$ by letting $g^{\prime}(v)$ be the vertex in $\tuple{\alg{A}, \struct{M}} $ obtained by (ii).

When $\tuple{\alg{A}, {\bf G}}$ is an empty structure we get that we can embed every finite $\alg{A}$-weighted graph in $\tuple{\alg{A}, \struct{M}}$, so the latter has the same age as the Fra\"iss\'e limit of the class of all finite $\alg{A}$-weighted graphs. Besides with (a) we can establish that the structure is also homogenous.\qed
\end{proof}

Our next example is the class of finite graded total orders.

\begin{Pro}$ \mathscr{K}_2^{\cong}$  is a Fra\"iss\'e class. \end{Pro}
\begin{proof} It is clear that $\mathscr{K}_2^{\cong}$ has HP.

We show next that $\mathscr{K}_2^{\cong}$ has AP. Consider  a v-formation $$\tuple{\alg{A}, \struct{M}_0} \subseteq \tuple{\alg{A}, \struct{M}_1}, \tuple{\alg{A}, \struct{M}_0} \subseteq \tuple{\alg{A}, \struct{M}_2}$$ in $\mathscr{K}_2^{\cong}$. We need to find $ \tuple{\alg{A}, \struct{M}_3} \in \mathscr{K}_2^{\cong}$ and embeddings $$\tuple{Id, f_1} \colon \tuple{\alg{A}, \struct{M}_1} \longrightarrow \tuple{\alg{A}, \struct{M}_3}$$ and $$\tuple{Id, f_2} \colon \tuple{\alg{A}, \struct{M}_2} \longrightarrow \tuple{\alg{A}, \struct{M}_3}$$ which coincide on their images for the elements of $\tuple{\alg{A}, \struct{M}_0}$. We let $M_3 = M_1 \cup M_2$ and keep the same order for elements of $M_0$ that we had in $\tuple{\alg{A}, \struct{M}_0}$ (i.e., the ordering relation takes exactly the same values for pairs of elements from $M_0$ as in $\tuple{\alg{A}, \struct{M}_0}$). Moreover with the ordering of $\tuple{\alg{A}, \struct{M}_0}$ paying attention to just the pairs of elements that take value $\1^\alg{A}$ in the order, we can list the elements of $M_0$ according to this order as $a_1, \dots, a_p$. Now consider the elements $x_0, \dots x_m \in M_1$ such that $\semvalue{ x_j < a_1 }^\alg{A}_{\bf M_1} = \1^\alg{A} \ (j \leq m)$ and 
$y_0, \dots y_n \in M_2$ such that $\semvalue{ y_j < a_1 }^\alg{A}_{\bf M_2} =  \1^\alg{A} \ (j \leq n)$. Now keep the values of the ordering for $\{x_0, \dots x_m, a_1, \dots, a_p\}$ that they had in $\tuple{\alg{A}, \struct{M}_1}$ and similarly for $\{y_0, \dots y_n, a_1, \dots, a_p\}$ and  $\tuple{\alg{A}, \struct{M}_2}$. On the other hand, put $\semvalue{x_i < y_j}^\alg{A}_{\bf M_3}= \1^\alg{A}$ where  $i \leq n$ and $ j \leq n$. Now consider $a_j$ and $a_{j+1}$ ($j\leq p-1$), and let $x^{\prime}_0, \dots x^{\prime}_m \in M_1$ be such that $\semvalue{ a_j < x^{\prime}_i < a_{j+1}}^\alg{A}_{\bf M_1} = \1^\alg{A} \ (i \leq m)$ and 
$y^{\prime}_0, \dots y^{\prime}_n \in M_2$ such that $\semvalue{  a_j < y^{\prime}_i < a_{j+1} }^\alg{A}_{\bf M_2} =  \1^\alg{A} \ (i \leq n)$. We again proceed as before keeping the evaluations from the original models $ \tuple{\alg{A}, \struct{M}_1}$ and  $\tuple{\alg{A}, \struct{M}_2}$ and adding that $\semvalue{x^{\prime}_i < y^{\prime}_j}^\alg{A}_{\bf M_3}= \1^\alg{A}$ where  $i \leq n$ and $ j \leq n$. Finally, take the elements $x_0, \dots x_m \in M_1$ such that $\semvalue{ a_p < x_j  }^\alg{A}_{\bf M_1} = \1^\alg{A} \ (j \leq m)$ and 
$y_0, \dots y_n \in M_2$ such that $\semvalue{ a_p  < y_j   }^\alg{A}_{\bf M_2} =  \1^\alg{A} \ (j \leq n)$ and evaluate in $\tuple{\alg{A}, \struct{M}_3}$ as we have been doing so far. Now by the linearity of the ordering, every element of $M_1$ and $M_2$ would have appeared at some point during our evaluation process, and since we kept the same evaluation from the original models   $ \tuple{\alg{A}, \struct{M}_1}$ and  $\tuple{\alg{A}, \struct{M}_2}$ when only elements from one of these two models were involved, the identity on $M_1$ or $M_2$ will gives us the desired embeddings.

From the above argument we can also deduce that $\mathscr{K}_2^{\cong}$ has the JEP.\qed
\end{proof}

We end with yet one more Fra\"iss\'e class of graded structures.

\begin{Pro}
$\mathscr{K}_3^{\cong}$  is a Fra\"iss\'e class.
\end{Pro}

\begin{proof} Clearly, $\mathscr{K}_3^{\cong}$ has HP. We show next that $\mathscr{K}_3^{\cong}$ has AP.

Consider a v-formation $$\tuple{\alg{A},\struct{M}_0} \subseteq \tuple{\alg{A}, \struct{M}_1},\,\,\,\,\,\,\, \tuple{\alg{A},\struct{M}_0} \subseteq \tuple{\alg{A},\struct{M}_2}$$ in $\mathscr{K}_3^{\cong}$. We need to find $\tuple{\alg{A},\struct{M}_3} \in \mathscr{K}_3^{\cong}$ and embeddings $$\tuple{Id, f_1}\colon \tuple{\alg{A},\struct{M}_1} \longrightarrow \tuple{\alg{A}, \struct{M}_3}$$ and $$\tuple{Id, f_2}\colon \tuple{\alg{A},\struct{M}_2} \longrightarrow \tuple{\alg{A}, \struct{M}_3}$$ which coincide on their images for the elements of $\tuple{\alg{A}, \struct{M}_0}$. We put $M_3 = M_1 \cup M_2$. Now let $a, b \in M_3$, there are a few possibilities:
\begin{itemize}
\item[] $a, b \in M_1$: $\semvalue{ a < b }^\alg{A}_{\struct{M}_3}= \semvalue{ a < b }^\alg{A}_{\struct{M}_1}$
\item[] $a, b \in M_1$: $\semvalue{ a < b }^\alg{A}_{\struct{M}_3}= \semvalue{ a < b }^\alg{A}_{\struct{M}_2}$
\item[]$a \in M_1 \setminus M_0$, $ b \in M_2 \setminus M_0$:   $\semvalue{ a < b }^\alg{A}_{\struct{M}_3} = \1^\alg{A}$ iff there is $x \in M_0$, $\semvalue{ a < x }^\alg{A}_{\struct{M}_1} \geq \1^\alg{A}$ and $\semvalue{  x < b }^\alg{A}_{\struct{M}_2} \geq \1^\alg{A}$ (otherwise, $\semvalue{ a < b }^\alg{A}_{\struct{M}_3} = \0^\alg{A}$); $\semvalue{  b < a }^\alg{A}_{\struct{M}_3} = \1^\alg{A}$ iff there is $x \in M_0$, $\semvalue{ a > x }^\alg{A}_{\struct{M}_1}  \geq \1^\alg{A}$ and $\semvalue{ x > b }^\alg{A}_{\struct{M}_2}  \geq \1^\alg{A}$  (otherwise, $\semvalue{ b < a }^\alg{A}_{\struct{M}_3} = \0^\alg{A}$).

\end{itemize}
The fact that $ \tuple{\alg{A}, \struct{M}_3} \in \mathscr{K}_3^{\cong}$ can be shown as follows. Obviously, $\semvalue{ a < a }^\alg{A}_{\struct{M}_3} \geq \1^\alg{A}$. For (3.2), given $a, b, c \in M_3$, we have to show that $\semvalue{ a < b }^\alg{A}_{\struct{M}_3}, \semvalue{ b < c }^\alg{A}_{\struct{M}_3} \geq \1^\alg{A}$ only if $  \semvalue{ a <  c }^\alg{A}_{\struct{M}_3} \geq \1^\alg{A}$. We take care of a few cases as an example. Suppose that $a \in M_1 \setminus M_0, b \in M_0$ and $c \in M_2 \setminus M_0$.  If $\semvalue{ a < b  }^\alg{A}_{\struct{M}_3} \geq \1^\alg{A}$ and $\semvalue{ b < c  }^\alg{A}_{\struct{M}_3} \geq \1^\alg{A}$, in fact, $\semvalue{ a < c }^\alg{A}_{\struct{M}_3} =  \1^\alg{A} $. Another interesting case here is when $a \in M_1, b \in M_1  \setminus M_0$ and $c \in M_2 \setminus M_0$. We may assume that  $\semvalue{a < b  }^\alg{A}_{\struct{M}_3} = \semvalue{ a < b  }^\alg{A}_{\struct{M}_1}\geq \1^\alg{A}$ and $\semvalue{b < c  }^\alg{A}_{\struct{M}_3}\geq \1^\alg{A}$, by construction of $ \tuple{\alg{A}, \struct{M}_3}$, then $\semvalue{ b < c  }^\alg{A}_{\struct{M}_3} = \1^\alg{A}$, so  there is $x \in M_0$, $\semvalue{ b < x }^\alg{A}_{\struct{M}_1} \geq \1^\alg{A}$ and $\semvalue{  x < c }^\alg{A}_{\struct{M}_2} \geq \1^\alg{A}$. But we have that $\semvalue{ a < x }^\alg{A}_{\struct{M}_1} \geq \1^\alg{A}$, so if $a \in M_1 \setminus M_0$, $\semvalue{ a <  c }^\alg{A}_{\struct{M}_3} = \1^\alg{A}$ and if $a \in M_0$, then since $\semvalue{ a < x }^\alg{A}_{\struct{M}_2}  = \semvalue{ a < x }^\alg{A}_{\struct{M}_1}  \geq \1^\alg{A}$, we must have that $\semvalue{ a < c }^\alg{A}_{\struct{M}_2} \geq \1^\alg{A}$.

Now, suppose for reductio that $\semvalue{ a < b }^\alg{A}_{\struct{M}_3}  \geq \1^\alg{A}$, $\semvalue{b < a }^\alg{A}_{\struct{M}_3}  \geq \1^\alg{A}$ and $a \neq b$. The only interesting case is when we assume w.l.o.g.\ that $a \in M_1 \setminus M_0$ and  $b \in M_2 \setminus M_0$. Then we have some $x \in M_0$ such that $\semvalue{ a < x }^\alg{A}_{M_1}  \geq \1^\alg{A}$
and  $\semvalue{b > x }^\alg{A}_{\struct{M}_2}  \geq \1^\alg{A}$. Furthermore, we have some $y \in M_0$ such that $\semvalue{ b < y }^\alg{A}_{\struct{M}_2}  \geq \1^\alg{A}$
and  $\semvalue{ a > y }^\alg{A}_{\struct{M}_1}  \geq \1^\alg{A}$. Then by (3.2) we may get that $\semvalue{ x > y }^\alg{A}_{\struct{M}_1} =   \semvalue{ x > y }^\alg{A}_{\struct{M}_2} \geq \1^\alg{A}$ and  $\semvalue{x < y }^\alg{A}_{\struct{M}_1} =   \semvalue{x < y }^\alg{A}_{M_2} \geq \1^\alg{A}$, which by (3.3) gives that $x=y$. But then  $\semvalue{ a < x }^\alg{A}_{\struct{M}_1}  \geq \1^\alg{A}$ and $\semvalue{ a > y }^\alg{A}_{\struct{M}_1}  \geq \1^\alg{A}$ imply by (3.3) that $a = x = y$. And similarly, $b = x = y$, so $a = b $, which is a contradiction.\qed
\end{proof}

\section{Conclusion}\label{5}

In this paper we have seen how one can export to the realm of graded structures the idea of constructing structures from smaller parts. Namely, we have shown how to adapt the argument for the construction of  Fra\"iss\'e limits to the setting of graded  or many-valued structures. 

It would be interesting to find further applications of the main result to meaningful classes of structures other than the examples studied in  this article. Furthermore, is it possible to find a limit construction that would preserve the property of safety of structures? We could probably add further conditions to the classes of structures in our main theorem to ensure safeness of the limit, but it be more desirable to find another construction that could work without the need for auxiliary properties.

\section*{Acknowledgments}
We are grateful to the anonymous referees for their corrections and remarks.
Guillermo Badia is supported by the project I 1923-N25 of the Austrian Science Fund (FWF). Carles Noguera is supported by the project GA17-04630S of the Czech Science Foundation (GA\v{C}R) and has also received funding from the European Union's Horizon 2020
research and innovation programme under the Marie Sklodowska-Curie grant
agreement No 689176 (SYSMICS project).

\end{document}